\documentclass{article}
\usepackage[english]{babel}
\usepackage{amsmath}
\usepackage{amsthm}
\usepackage{amssymb}
\usepackage{amscd}
\usepackage[all]{xy}
\usepackage[latin1]{inputenc}

\newtheorem{teor}{Theorem}[section]
\newtheorem{defi}{Definition}
\newtheorem{lema}[teor]{Lemma}
\newtheorem{prop}[teor]{Proposition}
\newtheorem{cor}[teor]{Corollary}
\newtheorem{rem}[teor]{Remark}

\newtheorem{ejems}[teor]{Examples}

\newtheorem{setting}[teor]{Setting}

\begin{document}

\title{On locally coherent hearts}

\author{Manuel Saorín \thanks{This work is backed by research projects from the Ministerio de Econom\'ia y Competitividad of Spain (MTM201346837-P) and the Fundaci\'on 'S\'eneca' of Murcia (19880/GERM/15), both with a part of FEDER funds. The author thanks these institutions for their support. He also thanks Carlos Parra and the referee for their comments and the careful reading of the paper.}\\ Departamento de Matemáticas\\
Universidad de Murcia, Aptdo. 4021\\
30100 Espinardo, Murcia\\
SPAIN\\ {\it msaorinc@um.es}}

\date{}

\maketitle


\begin{abstract}

{\bf Let $\mathcal{G}$ be a locally coherent Grothendieck category. We show that, under particular conditions, if a t-structure $\tau$ in the unbounded derived category $\mathcal{D}(\mathcal{G})$ restricts to the bounded derived category $\mathcal{D}^b(fp(\mathcal{G}))$ of its category of finitely presented (=coherent) objects, then its heart $\mathcal{H}_\tau$  is a locally coherent Grothendieck category on which $\mathcal{H}_\tau\cap\mathcal{D}^b(fp(\mathcal{G}))$ is the class of finitely presented objects. Those particular conditions are always satisfied when $\mathcal{G}$ is arbitrary and $\tau$ is the Happel-Reiten-Smalo t-structure in $\mathcal{D}(\mathcal{G})$ associated to a torsion pair in $fp(\mathcal{G})$ or when $\mathcal{G}=\text{Qcoh}(\mathbb{X})$ is the category of quasicoherent sheaves on a noetherian affine scheme $\mathbb{X}$  and $\tau$ is any compactly generated t-structure in $\mathcal{D}(\mathbb{X}):=\mathcal{D}(\text{Qcoh}(\mathbb{X}))$ which restrict to $\mathcal{D}^b(\mathbb{X}):=\mathcal{D}^b(\text{coh}(\mathbb{X}))$. In particular, the heart of any t-structure in $\mathcal{D}^b(\mathbb{X})$ is the category of finitely presented objects of a locally coherent Grothendieck category.  
 }
\end{abstract}

{\bf Mathematics Subjects Classification:} 18E30, 18E15, 13DXX, 14XX, 16EXX.

\section{Introduction}

Beilinson, Bernstein and Deligne \cite{BBD} introduced the notion of
a t-structure in a triangulated category in their study of perverse
sheaves on an algebraic or analytic variety. If $\mathcal{D}$ is
such a triangulated category, a t-structure is a pair of full
subcategories satisfying some axioms  which guarantee that their intersection is an
abelian category $\mathcal{H}$, called the heart of the t-structure.
This category  comes with a cohomological functor
$\mathcal{D}\longrightarrow\mathcal{H}$. Roughly speaking, a
t-structure allows to develop an intrinsic (co)homology theory,
where the homology 'spaces' are again objects of $\mathcal{D}$
itself.

T-structures are nowadays used in several branches of Mathematics, with special impact in Algebraic Geometry, Homotopical Algebra and Representation Theory of Groups and Algebras. When dealing with t-structures, a natural question arises. It asks under which conditions the heart of a given t-structure is a 'nice' abelian category. Using a classical hierarchy for abelian categories introduced by Grothendieck, one may think of Grothendieck and module categories as the nicest possible abelian categories. It is then not surprising that   the question of when the heart of a t-structure is a Grothendieck or module category deserved much attention in recent times (see, e.g.,  \cite{HKM}, \cite{CGM}, \cite{CG}, \cite{CMT}, \cite{MT}, \cite{PS1}, \cite{PS2}, \cite{PS4}, \cite{PV}, \cite{NSZ}). 

Among Grothendieck categories, the most studied ones are those that have  finiteness conditions (e.g. those which are locally coherent, locally noetherian or even locally finite). Module categories over a noetherian or coherent rings or over Artin algebras, or the categories of quasi-coherent sheaves over coherent or noetherian schemes provide examples of such categories.  A natural subsequent question would ask when a given t-structure has a heart which is a Grothendieck category with good finiteness conditions. In this paper, we tackle the question for the locally coherent condition,  assuming that the t-structure lives in the (unbounded) derived category $\mathcal{D}(\mathcal{G})$ of a Grothendieck category $\mathcal{G}$ which is itself locally coherent. Although to find a general answer seems to be hopeless, it is not so when the t-structure restricts to $\mathcal{D}^b(\text{fp}(\mathcal{G}))$, the bounded derived category of the category of finitely presented (=coherent) objects. Our basic technical result in the paper, Proposition \ref{prop.locally coherent Grothendieck heart},  gives a precise list of sufficient conditions on a t-structure in $\mathcal{D}(\mathcal{G})$ so that its heart $\mathcal{H}$ is a locally coherent Grothendieck category on which $\mathcal{H}\cap\mathcal{D}^b(\text{fp}(\mathcal{G}))$ is the class of its finitely presented objects.  As an appication, we get the  main results of the paper, referring the reader to next section for the notation and terminology used:

\begin{enumerate}
\item (Theorem \ref{teor.HRS}) Let $\mathcal{G}$ be a locally coherent Grothendieck category and $\mathbf{t}=(\mathcal{T},\mathcal{F})$ be a torsion pair in $\mathcal{G}$. The associated Happel-Reiten-Smal$\emptyset$ t-structure in $\mathcal{D}(\mathcal{G})$ restricts to $\mathcal{D}^b(\text{fp}(\mathcal{G}))$ and has a heart which is locally coherent Grothendieck category if, and only if, $\mathcal{F}$ is closed under taking direct limits in $\mathcal{G}$ and $\mathbf{t}$ restricts to $\text{fp}(\mathcal{G})$. 

\item (Theorem \ref{teor.main}) If $R$ is a commutative noetherian ring, then any compactly generated t-structure in $\mathcal{D}(R)$ which restricts to $\mathcal{D}^b_{fg}(R)\cong \mathcal{D}^b(R-\text{mod})$  has a heart $\mathcal{H}$ which is a locally coherent Grothendieck category on which $\mathcal{H}\cap\mathcal{D}^b_{fg}(R)$ is the class of its finitely presented objects. 

\item (Corollary \ref{cor.realisation}) If $R$ is a commutative noetherian ring, then the heart of each t-structure in $\mathcal{D}_{fg}^b(R)$ is equivalent to the category of finitely presented objects of some locally coherent Grothendieck category.  
\end{enumerate}
Of course, when taking the affine scheme $\mathbb{X}=\text{Spec}(R)$ in 2 and 3, one obtains the geometric versions mentioned in the abstract (see also Corollary \ref{cor.geometric translation}). 

The organization of the paper goes as follows. Section 2 introduces all the concepts and terminology used in the paper. In Section 3 we give some general results about locally coherent Grothendieck categories which are used later. Section 4 contains the  technical Proposition \ref{prop.locally coherent Grothendieck heart}, which is basic for the paper, and a few auxiliary results needed for its proof. Section 5 is dedicated to the Happel-Reiten-Smal$\emptyset$ t-structure and the proof of Theorem \ref{teor.HRS}. The final Section 6 gives Theorem \ref{teor.main}, of which Corollary \ref{cor.realisation} is a direct consequence, and two lemmas needed for its proof.

\section{Preliminaries and terminology}

All categories in this paper will be additive and all rings will be supposed to be associative with unit, unless otherwise specified. Whenever the term 'module' is used over a noncommutative ring, it will mean 'left module' and, for a given ring $R$, we will denote by $R-\text{Mod}$ the category of all $R$-modules.  Let $\mathcal{A}$ be  an additive category in the rest of the paragraph. If $\mathcal{C}$ is any class of objects in $\mathcal{A}$, the symbol $\mathcal{C}^\perp$ (resp. ${}^\perp\mathcal{C}$) will denote the full subcategory of $\mathcal{A}$ whose objects are those $X\in\text{Ob}(\mathcal{A})$ such that $\text{Hom}_\mathcal{A}(C,X)=0$ (resp. $\text{Hom}_\mathcal{A}(X,C)=0$), for all $C\in\mathcal{C}$.  The expression '$\mathcal{A}$ has  products (resp. coproducts)' will mean that $\mathcal{A}$  has arbitrary set-indexed products (resp. coproducts). When $\mathcal{S}$ is a set of objects in $\mathcal{A}$, we shall denote by $\text{sum}(\mathcal{S})$ (resp. $\text{add}(\mathcal{S})$) the class of objects which are isomorphic to  a finite coproduct (resp. a direct summand of a finite coproduct) of objects of $\mathcal{S}$. When $\mathcal{A}$ has  coproducts, we shall say that an object $X$ is a
\emph{compact (or small) object} when the functor
$\text{Hom}_\mathcal{A}(X,?):\mathcal{A}\longrightarrow\text{Ab}$
preserves coproducts. 

Two types of additive categories will get most of our interest in this paper. The first one is that of \emph{abelian categories} (see \cite{Po}) and the second one is that of \emph{triangulated categories} (see \cite{N}). Diverting from the terminology in this latter reference, for a triangulated category $\mathcal{D}$, the shift or suspension functor  will be denoted by $?[1]$, putting $?[k]$ for its $k$-th power, for each $k\in\mathbb{Z}$. We shall use the term \emph{class (resp. set) of generators} with two different meanings, depending on whether we are in the abelian or the triangulated context. When $\mathcal{A}$ is an abelian category with coproducts, a class (resp. set) of generators $\mathcal{S}$ is a class (resp. set) of objects such that each object in $\mathcal{A}$ is an epimorphic image of a coproduct of objects in $\mathcal{S}$. When $\mathcal{S}$ is a class (resp. set) of objects in the triangulated category $\mathcal{D}$, we shall say that it is a class (resp. set) of generators in case an object $X$ of $\mathcal{D}$ is zero exactly when $\text{Hom}_\mathcal{D}(S[k],X)=0$, for all $S\in\mathcal{S}$ and all $k\in\mathbb{Z}$. 

Given a triangulated category $\mathcal{D}$, a subcategory $\mathcal{E}$ will be called a \emph{triangulated subcategory} when it is closed under taking extensions and $\mathcal{E}[1]=\mathcal{E}$. If, in addition, it is closed under taking direct summands, we will say that $\mathcal{E}$  is a \emph{thick subcategory} of $\mathcal{D}$. When $\mathcal{S}$ is a set of objects of $\mathcal{D}$, we shall denote by $\text{tria}_\mathcal{D}(\mathcal{S})$ (resp. $\text{thick}_\mathcal{D}(\mathcal{S})$) the smallest triangulated (resp. thick) subcategory of $\mathcal{D}$ which contains $\mathcal{S}$. 

For an additive  category $\mathcal{A}$, we will denote by
$\mathcal{C}(\mathcal{A})$ and  $\mathcal{K}(\mathcal{A})$  the category of chain complexes of
objects of $\mathcal{A}$ and the homotopy category of $\mathcal{A}$. Diverting from the classical notation, we will write superindices for chains, cycles and boundaries in ascending order.  We will denote by $\mathcal{C}^-(\mathcal{A})$ (resp. $\mathcal{K}^-(\mathcal{A})$), $\mathcal{C}^+(A)$ resp. $\mathcal{K}^+(\mathcal{A})$) and $\mathcal{C}^b(\mathcal{A})$ (resp. $\mathcal{K}^b(\mathcal{A})$) the full subcategories of $\mathcal{C}(\mathcal{A})$ (resp. $\mathcal{K}(\mathcal{A})$) consisting of those objects isomorphic to upper bounded, lower bounded and (upper and lower) bounded complexes, respectively. Note that $\mathcal{K}(\mathcal{A})$ is always a triangulated category of which $\mathcal{K}^-(\mathcal{A})$, $\mathcal{K}^+(\mathcal{A})$ and $\mathcal{K}^b(\mathcal{A})$ are triangulated subcategories. When $\mathcal{A}$ is an abelian category, we will denote by
$\mathcal{D}(\mathcal{A})$ its \emph{derived category}, which is the one obtained from $\mathcal{C}(\mathcal{A})$ by (keeping the same objects and) formally inverting the quasi-isomorphisms (see \cite{V} for the details). We shall denote by $\mathcal{D}^-(\mathcal{A})$ (resp. $\mathcal{D}^+(\mathcal{A})$, resp. $\mathcal{D}^b(\mathcal{A})$) the full subcategory of $\mathcal{D}(\mathcal{A})$ consisting of those complexes $X$ such that $H^k(X)=0$, for all $k\gg 0$ (resp. $k\ll 0$, resp. $|k| \gg 0$), where $H^k:\mathcal{D}(\mathcal{A})\longrightarrow\mathcal{A}$ denotes the $k$-th homology functor, for each $k\in\mathbb{Z}$. The objects of $\mathcal{D}^{-}(\mathcal{A})$ (resp. $\mathcal{D}^{+}(\mathcal{A})$, resp. $\mathcal{D}^{b}(\mathcal{A})$) will be called \emph{homologically upper bounded (resp. homologically lower bounded, resp. homologically bounded) complexes}. For integers $m\leq n$, we will denote by $\mathcal{D}^{[m,n]}(\mathcal{A})$  the full subcategory of $\mathcal{D}(\mathcal{A})$ consisting of the complexes $X$ such that $H^k(X)=0$ for integers $k$ not in the closed interval $[m,n]$. We will also use $\mathcal{D}^{\leq n}(\mathcal{A})$ (resp. $\mathcal{D}^{< n}(\mathcal{A})$) and $\mathcal{D}^{\geq n}(\mathcal{A})$ (resp. $\mathcal{D}^{> n}(\mathcal{A})$) to denote the full subcategories consisting of the complexes $X$ such that $H^i(X)=0$, for all $i>n$ (resp. $i\geq n$) and all $i<n$ (resp. $i\leq n$), respectively. 

Strictly speaking, for a general abelian category $\mathcal{A}$, the category $\mathcal{D}(\mathcal{A})$ need not exist since the morphisms between two given objects could form a proper class and not just a set. However, this problem disappears when $\mathcal{A}=\mathcal{G}$ is a \emph{Grothendieck category}. This is a cocomplete abelian category with a set of generators on which direct limits are exact. In  a Grothendieck category $\mathcal{G}$ an object $S$ is called \emph{finitely presented} when $\text{Hom}_\mathcal{G}(S,?):\mathcal{G}\longrightarrow\text{Ab}$ preserves direct limits. We say that $\mathcal{G}$ is  \emph{locally finitely presented} when it has a set of finitely presented generators. The reader is referred to \cite{CB} for the corresponding more general concept of locally finitely presented additive categories with direct limits and is invited  to check by her/himself that, in the case of Grothendieck categories, it coincides with the one given here. Recall that an object in a Grothendieck category is called \emph{noetherian} when it satisfies the ascending chain condition on subobjects. A \emph{locally noetherian} Grothendieck category is a Grothendieck category which has a set of noetherian generators. When $\mathcal{G}$ is locally finitely presented and locally noetherian,  an object $N$ of $\mathcal{G}$ is noetherian if, and only if, it is finitely presented (see \cite[Proposition A.11]{Kr} for one direction, the reverse one being obvious since each noetherian object in such a category is an epimorphic image of a finitely presented one and the kernel of this epimorphism is again noetherian). 

Recall that if $\mathcal{D}$ and $\mathcal{A}$ are a triangulated
and an abelian category, respectively, then an additive  functor
$H:\mathcal{D}\longrightarrow\mathcal{A}$ is a \emph{cohomological
functor} when, given any triangle $X\longrightarrow Y\longrightarrow
Z\stackrel{+}{\longrightarrow}$, one gets an induced long exact
sequence in $\mathcal{A}$:

\begin{center}
$\cdots \longrightarrow H^{n-1}(Z)\longrightarrow H^n(X)\longrightarrow
H^n(Y)\longrightarrow H^n(Z)\longrightarrow
H^{n+1}(X)\longrightarrow \cdots$,
\end{center}
where $H^n:=H\circ (?[n])$, for each $n\in\mathbb{Z}$. 

A \emph{torsion pair} in the abelian category $\mathcal{A}$ is a pair $\mathbf{t}=(\mathcal{T},\mathcal{F})$ of full subcategories such that $\text{Hom}_\mathcal{A}(T,F)=0$, for all $T\in\mathcal{T}$ and $F\in\mathcal{F}$, and each object $X$ of $\mathcal{A}$ fits into an exact sequence $0\rightarrow T_X\longrightarrow X\longrightarrow F_X\rightarrow 0$, where $T_X\in\mathcal{T}$ and $F_X\in\mathcal{F}$. In this latter case the assignments $X\rightsquigarrow T_X$ and $X\rightsquigarrow F_X$ extend to endofunctors $t, (1:t):\mathcal{A}\longrightarrow\mathcal{A}$. The functor $t$ is usually called the \emph{torsion radical} associated to $\mathbf{t}$. The torsion pair $\mathbf{t}$ will be called \emph{hereditary} when $\mathcal{T}$ is closed under taking subobjects in $\mathcal{A}$.

Let now $\mathcal{D}$ be a triangulated category. A
\emph{t-structure} in $\mathcal{D}$ (see \cite[Section 1]{BBD}) is a pair
$\tau =(\mathcal{U},\mathcal{W})$ of full subcategories, closed under
taking direct summands in $\mathcal{D}$,  which satisfy the
 following  properties:

\begin{enumerate}
\item[i)] $\text{Hom}_\mathcal{D}(U,W[-1])=0$, for all
$U\in\mathcal{U}$ and $W\in\mathcal{W}$;
\item[ii)] $\mathcal{U}[1]\subseteq\mathcal{U}$;
\item[iii)] For each $X\in Ob(\mathcal{D})$, there is a triangle $U\longrightarrow X\longrightarrow
V\stackrel{+}{\longrightarrow}$ in $\mathcal{D}$, where
$U\in\mathcal{U}$ and $V\in\mathcal{W}[-1]$.
\end{enumerate}
In this case $\mathcal{W}=\mathcal{U}^\perp
[1]$ and $\mathcal{U}={}^\perp (\mathcal{W}[-1])={}^\perp
(\mathcal{U}^\perp )$ and,  for this reason, we will write a t-structure
as $\tau =(\mathcal{U},\mathcal{U}^\perp [1])$. We will call $\mathcal{U}$
and $\mathcal{U}^\perp$ the \emph{aisle} and the \emph{co-aisle} of
the t-structure. The objects $U$ and $V$ in the above
triangle are uniquely determined by $X$, up to isomorphism, and
define functors
$\tau_\mathcal{U}:\mathcal{D}\longrightarrow\mathcal{U}$ and
$\tau^{\mathcal{U}^\perp}:\mathcal{D}\longrightarrow\mathcal{U}^\perp$
which are right and left adjoints to the respective inclusion
functors. We call them the \emph{left and right truncation functors}
with respect to the given t-structure.  The full subcategory
$\mathcal{H}=\mathcal{U}\cap\mathcal{W}=\mathcal{U}\cap\mathcal{U}^\perp
[1]$ is called the \emph{heart} of the t-structure and it is an
abelian category, where the short exact sequences 'are' the
triangles in $\mathcal{D}$ with its three terms in $\mathcal{H}$.
Moreover, with the obvious abuse of notation,  the assignments
$X\rightsquigarrow (\tau_{\mathcal{U}}\circ\tau^{\mathcal{U}^\perp
[1]})(X)$ and $X\rightarrow (\tau^{\mathcal{U}^\perp
[1]}\circ\tau_\mathcal{U})(X)$ define   naturally isomorphic
functors $\mathcal{D}\longrightarrow\mathcal{H}$ whih are
cohomological (see \cite{BBD}). We will identify them and denote the corresponding functor by $\tilde{H}$. When $\mathcal{D}$ has coproducts, the t-structure $\tau$ will be called \emph{compactly generated} when there is a set $\mathcal{S}\subseteq\mathcal{U}$, formed by compact objects in $\mathcal{D}$, such that $\mathcal{W}[-1]=\mathcal{U}^\perp$ consists of the objects $Y$ such that $\text{Hom}_\mathcal{D}(S[k],Y)=0$, for all $S\in\mathcal{S}$ and integers $k\geq 0$. 

When $\mathcal{D}$ is a triangulated category with coproducts,  we will use the term \emph{Milnor colimit} of a sequence of morphisms $X_0\stackrel{x_1}{\longrightarrow}X_1\stackrel{x_2}{\longrightarrow}\cdots \stackrel{x_n}{\longrightarrow}X_n\stackrel{x_{n+1}}{\longrightarrow} \cdots $ what in \cite{N} is called homotopy colimit. It will be denoted $\text{Mcolim}(X_n)$, without reference to the $x_n$.

\section{Generalities about  locally coherent Grothendieck categories}

In this section we are interested in a particular case of locally finitely presented Grothendieck categories. Let us start by the following result which is folklore.

\begin{lema} \label{lem.abelian exact subcategory}
Let $\mathcal{A}$ be an abelian category and $\mathcal{B}$ be a full additive subcategory. The following assertions are equivalent:

\begin{enumerate}
\item $\mathcal{B}$ is an abelian category such that the inclusion functor $\mathcal{B}\hookrightarrow\mathcal{A}$ is exact;
\item $\mathcal{B}$ is closed under taking finite (co)products, kernels and cokernels in $\mathcal{A}$.
\end{enumerate}
In this case  we will say that $\mathcal{B}$ is an \emph{abelian exact subcategory} of $\mathcal{A}$.
\end{lema}

Note that if $\mathcal{G}$ is a locally finitely presented Grothendieck category, then the class $fp(\mathcal{G})$ of finitely presented objects is skeletally small and is closed under taking cokernels and finite coproducts. 

\begin{defi} \label{def.locally coherent Grothendieck}
A Grothendieck category $\mathcal{G}$ is called \emph{locally coherent} when it is locally finitely presented and the subcategory $fp(\mathcal{G})$  is an abelian exact subcategory of $\mathcal{G}$ (equivalently, when $fp(\mathcal{G})$ is closed under taking kernels). 
\end{defi}

Recall that a \emph{pseudo-kernel} (resp. \emph{pseudo-cokernel}) of a morphism $f:X\longrightarrow Y$ in the additive category $\mathcal{A}$ is a morphism $u:Z\longrightarrow X$ (resp. $v:Y\longrightarrow Z$) such that the sequence of contravariant (resp. covariant) functor $\text{Hom}_\mathcal{A}(?,Z)\stackrel{u_*}{\longrightarrow}\text{Hom}_\mathcal{A}(?,X)\stackrel{f_*}{\longrightarrow}\text{Hom}_\mathcal{A}(?,Y)$ is exact (resp. $\text{Hom}_\mathcal{A}(Z,?)\stackrel{v^*}{\longrightarrow}\text{Hom}_\mathcal{A}(Y,?)\stackrel{f_*}{\longrightarrow}\text{Hom}_\mathcal{A}(X,?)$) is exact. We say that $\mathcal{A}$ \emph{has pseudo-kernels (resp. pseudo-cokernels)} when each morphism in $\mathcal{A}$ has a pseudo-kernel (resp. pseudo-cokernel). 

\begin{ejems}
The following are examples of locally coherent Grothendieck categories to which the results in this and next section apply:

\begin{enumerate}
\item $R-\text{Mod}$, when $R$ is a left coherent ring $R$ (i.e. when each finitely generated left ideal of $R$ is finitely presented).
\item The category $[\mathcal{C},Ab]$  (resp. $[\mathcal{C}^{op},Ab]$) of covariant (resp. contravariant)   additive functors $\mathcal{C}\longrightarrow Ab$, where $\mathcal{C}$ is a (skeletally) small additive category with pseudo-cokernels (resp. pseudo-kernels). In particular, when $\mathcal{C}$ is a skeletally small abelian or triangulated category, both  $[\mathcal{C},Ab]$  and $[\mathcal{C}^{op},Ab]$ are locally coherent Grothendieck categories.  
\item The category $\text{Qcoh}(\mathbb{X})$ of quasi-coherent sheaves, where $\mathbb{X}$ is a \emph{coherent scheme}, i.e., a quasi-compact and quasi-separated scheme admitting a covering $\mathbb{X}=\bigcup_{i\in I}U_i$ by affine open subschemes $U_i$ such that $U_i=\text{Spec}(A_i)$, for a commutative coherent ring $A_i$, for each $i\in I$. 
\item Any locally noetherian and locally finitely presented Grothendieck category.
\end{enumerate}
\end{ejems}
\begin{proof}
Example 1 is well-known and the covariant version of example 2 follows from \cite[Propositions 1.3 and 2.1]{He}, taking into account that, in Herzog's proposition 2.1, the proof that each representable functor $(X,?)$ is a coherent object only requires that each morphism $X\longrightarrow Y$ has a pseudo-cokernel. The contravariant version of assertion 2 follows by duality. For example 3, see \cite[Proposition 40]{Ga} (see  also \cite[Example 1.1.6.iv]{Si}). Finally, example 4 is clear since $fp(\mathcal{G})$ coincides with the class of noetherian objects in that case, and this latter class is always closed under taking kernels (even subobjects).  
\end{proof}

\begin{lema} \label{lem.bounded f.p. homology}
Let $\mathcal{G}$ be a locally coherent Grothendieck category, let $\mathcal{S}$ be a set of finitely presented generators of $\mathcal{G}$ and let $M$ be any object in $\mathcal{D}(\mathcal{G})$. The following assertions hold:

\begin{enumerate}
\item $M$ is a homologically upper bounded complex whose homology objects are finitely presented if, and only if, $M$ is isomorphic in $\mathcal{D}(\mathcal{G})$ to an upper bounded complex $N$ of objects in $\text{sum}(\mathcal{S})$. Moreover, $N$ can be chosen so that $\text{max}\{i\in\mathbb{Z}\text{: }N^i\neq 0\}=\text{max}\{i\in\mathbb{Z}\text{: }H^i(M)\neq 0\}$.
\item $M$ is homologically bounded and its homology objects are finitely presented if, and only if, $M$ is isomorphic in $\mathcal{D}(\mathcal{G})$ to a bounded complex 
$$\cdots 0\longrightarrow N^{m}\longrightarrow N^{m+1}\longrightarrow N^{n-1}\longrightarrow N^n\longrightarrow 0\cdots ,$$
where the $N^i$ are finitely presented objects (and $N^i\in\text{sum}(\mathcal{S})$, for $m<i\leq n$).

If, moreover,  the objects of $\mathcal{S}$ form a set of compact generators of $\mathcal{D}(\mathcal{G})$, then also the following assertion holds:

\item The compact objects of $\mathcal{D}(\mathcal{G})$ are those isomorphic to direct summands of  bounded complexes of objects in $\text{add}(\mathcal{S})$.
\end{enumerate}
\end{lema}
\begin{proof}
We will frequently use the fact that if $M$ is a complex whose homology objects are all finitely presented, then a given $k$-cycle object $Z^k=Z^k(M)$ is finitely presented if and only if so is the $k$-boundary object $B^k=B^k(M)$. 

1) The proof of this assertion is reminiscent of the dual of the proof of Lemma 4.6(3) in \cite[Chapter I]{Har}, with $\mathcal{A}'=fp(\mathcal{G})$ and $\mathcal{A}=\mathcal{G}$, although the assumptions of that lemma do not hold in our situation. By truncating at the greatest integer $i$ such that $H^i(M)\neq 0$ and shifting if necessary,  we can assume without loss of generality that $M$ is concentrated in degrees $\leq 0$ and that $H^0(M)\neq 0$. We then inductively construct a sequence in $\mathcal{C}(\mathcal{G})$  $$...M_n\stackrel{f_n}{\longrightarrow}M_{n-1}\longrightarrow ...\longrightarrow M_1\stackrel{f_1}{\longrightarrow}M_0\stackrel{f_0}{\longrightarrow}M$$ satisfying the following properties:

\begin{enumerate}
\item[a)] Each $M_n$ is concentrated in degrees $\leq 0$;
\item[b)] The connecting chain maps $f_n:M_n\longrightarrow M_{n-1}$ are quasi-isomorphisms, for all $n\in\mathbb{N}$ (with the convention that $M_{-1}=M$);
\item[c)] Given $n\in\mathbb{N}$, one has that $M_n^{-k}\in\text{sum}(\mathcal{S})$ for $0\leq k\leq n$;
\item[d)] Given any $k\in\mathbb{N}$, the morphism $f_n^{-k}:M_n^{-k}\longrightarrow M_{n-1}^{-k}$ is the identity map, for all $n>k$.
\end{enumerate}
Once the sequence has been constructed, we clearly see that the inverse limit of the sequence, $X:=\varprojlim_{\mathcal{C}(\mathcal{G})}(M_n)$,  is a complex of objects in $\text{sum}(\mathcal{S})$ concentrated in degrees $\leq 0$ such that the induced chain map $X\longrightarrow M$ is a quasi-isomorphism. 

We now pass to construct the mentioned sequence. At step $0$, one easily gives a morphism $f:X^0\longrightarrow M^0$ such that $X^0\in\text{add}(\mathcal{S})$ and the composition $X^0\stackrel{f}{\longrightarrow}M^0\stackrel{p}{\longrightarrow}H^0(M)$ is an epimorphism, where $p$ is the projection.  Now, taking the pullback of $f$ and the differential $M^{-1}\longrightarrow M^0$, we easily get a quasi-isomorphism $f_0:M_0\longrightarrow M$, where $f_0^{-k}:M_0^{-k}=M^{-k}\longrightarrow M^{-k}$ is the identity map for all $k\geq 2$, and $f_0^0:M_0^0=X^0\longrightarrow M^0$ is $f$. 

Assume now that $n>0$ and that the quasi-isomorphisms $M_{n-1}\stackrel{f_{n-1}}{\longrightarrow}M_{n-2}\longrightarrow \cdots \longrightarrow M_{1}\stackrel{f_1}{\longrightarrow}M_{0}\stackrel{f_0}{\longrightarrow}M$ have already been constructed, satisfying the requirements. Note that  $Z^{-k}:=Z^{-k}(M_{n-1})$, and hence also $B^{-k}:=B^{-k}(M_{n-1})$, is a finitely presented object for $k=0,1,...,n-1$. 
   Let us fix a direct system $(Y_i)_{i\in I}$ in $\text{fp}(\mathcal{G})$ such that $\varinjlim Y_i\cong M_{n-1}^{-n}$. Replacing the directed set $I$ by a cofinal subset if necessary, there is no loss of generality in assuming that the composition $Y_j\stackrel{u_j}{\longrightarrow}\varinjlim Y_i\cong M_{n-1}^{-n}\stackrel{d^{-n}}{\longrightarrow}B^{-n+1}$ is an epimorphism, for all $j\in I$, where $u_j$ is the canonical morphism to the direct limit. It is seen in a straightforward way that we have a direct system of exact sequences  $$0\rightarrow u_i^{-1}(Z^{-n})\longrightarrow Y_i\stackrel{d^{-n}\circ u_i}{\longrightarrow}B^{-n+1}\rightarrow 0\hspace*{0.5cm} (i\in I)$$ whose direct limit is precisely the canonical exact sequence $0\rightarrow Z^{-n}\longrightarrow X^{-n}\stackrel{d^{-n}}{\longrightarrow}B^{-n+1}\rightarrow 0$. Due to the fact that $H^{-n}:=H^{-n}(M_{n-1})$ is finitely presented, there is some index $j\in I$ such that the composition $ u_j^{-1}(Z^{-n})\stackrel{u_j}{\longrightarrow}Z^{-n}\stackrel{p}{\longrightarrow}H^{-n}$ is an epimorphism. We fix such an index $j$ and choose any epimorphism $\epsilon:X^{-n}\twoheadrightarrow Y_j$, with $X^{-n}\in\text{sum}(\mathcal{S})$. Putting $M_n^{-n}:=X^{-n}$, the composition $g:M_n^{-n}\stackrel{\epsilon}{\longrightarrow}Y_j\stackrel{u_j}{\longrightarrow}\varinjlim (Y_j)\cong M_{n-1}^{-n}$ is then a morphism which leads to the following commutative diagram, where all squares are bicartesian:

$$\xymatrix{M_{n}^{-n-1} \ar[r] \ar[d]_{g' \vspace{0.1 cm}} & \tilde{B}^{-n} \hspace{0.1 cm} \ar@{^(->}[r] \ar[d] & g^{-1}(Z^{-n})  \hspace{0.1 cm} \ar@{^(->}[r]  \ar[d] & M_{n}^{-n} \ar[d]^{g} \\ M_{n-1}^{-n-1} \ar[r]^{\hspace{0.4 cm}d^{-n-1}} & B^{-n} \hspace{0.1 cm} \ar@{^(->}[r] & Z^{-n}\hspace{0.1 cm} \ar@{^(->}[r] & M_{n-1}^{-n}}$$

We easily derive a quasi-isomorphism  $h:M_n\longrightarrow M_{n-1}$, where $h^{-k}:M_n^{-k}=M_{n-1}^{-k}\longrightarrow M_{n-1}^{-k}$ is the identity map, for $k\geq 0$ and $k\neq n, n+1$, and $h^{-n-1}=g'$ and $h^{-n}=g$ are the morphisms from the last diagram. 

2) By assertion 1, we can assume that $M$ is of the form 

$$\cdots \longrightarrow N^{k}\longrightarrow N^{k+1}\longrightarrow \cdots \longrightarrow N^{n-1}\longrightarrow N^n\longrightarrow 0 \cdots ,$$
where the $N^i$ are in $\text{sum} (\mathcal{S})$. Let us assume that $m=\text{min}\{j\in\mathbb{Z}\text{: }H^j(M)\neq 0\}$. Then the intelligent truncation at $m$ gives  the complex

$$\tau^{\geq m}M: \cdots 0\longrightarrow B^m\hookrightarrow N^{m}\longrightarrow N^{m+1}\longrightarrow \cdots  \longrightarrow N^{n-1}\longrightarrow N^n\longrightarrow 0 \cdots ,$$
where $B^m$ is $m$-boundary object of $M$. But $B^m$ is finitely presented since so is $Z^m=\text{Ker}(N^m\longrightarrow N^{m+1})$. We then take $N^m=B^m$ and the proof of the implication is complete because the canonical map $\tau^{\geq m}M\longrightarrow M$ is an isomorphism in $\mathcal{D}(\mathcal{G})$.

In the rest of the proof, we assume that $\mathcal{S}$ is a set of compact generators of  $\mathcal{D}(\mathcal{G})$. 

3) Note that each bounded complex of objects in $\text{add}(\mathcal{S})$ is compact in $\mathcal{D}(\mathcal{G})$ since it is a finite iterated extension of stalks $X[k]$, with $X\in\text{add}(\mathcal{S})$. Conversely, suppose that $M$ is a compact object in $\mathcal{D}(\mathcal{G})$. It follows from \cite[Theorem 5.3]{Ke} that it is a direct summand of a finite iterated extension of complexes of the form $S[k]$, with $S\in\mathcal{S}$ and $k\in\mathbb{Z}$. In particular $M$ has bounded and finitely presented homology. If we fix now a quasi-isomorphism $f:P\longrightarrow M$ such that $P$ is a bounded above complex of objects in $\text{add}(\mathcal{S})$, then we can assume without loss of generality that $P^0\neq 0=P^k$, for all $k>0$. Note that then $P$ is the Milnor  colimit of the stupid truncations $\sigma^{\geq -n}P$ . Since $P$ is compact in $\mathcal{D}(\mathcal{G})$, an argument as in the proof of \cite[Theorem 5.3]{Ke} shows that the identity map $1_P$ factors in the form $P\longrightarrow\sigma^{\geq -n}P\longrightarrow P$, for some $n\in\mathbb{N}$. It follows that $M\cong P$ is isomorphic in $\mathcal{D}(\mathcal{G})$ to a direct summand of a  bounded complex of objects in $\text{add}(\mathcal{S})$. 
\end{proof}

When $\mathcal{G}$ is a locally coherent Grothendieck category, one easily gets from assertion 1 and 2 of last lemma that $\mathcal{D}^b(\text{fp}(\mathcal{G}))$ is equivalent, as a triangulated category, to the full subcategory $\mathcal{D}^b_{fp}(\mathcal{G})$ of $\mathcal{D}(\mathcal{G})$ consisting of those complexes $M\in\mathcal{D}^b(\mathcal{G})$ such that $H^i(M)\in\text{fp}(\mathcal{G})$, for all $i\in\mathbb{Z}$. In the sequel we will identify these equivalent triangulated categories, viewing $\mathcal{D}^b(\text{fp}(\mathcal{G}))$ as a full triangulated subcategory of $\mathcal{D}(\mathcal{G})$.  

\begin{defi} \label{def.fp-injective}
Let $\mathcal{G}$ be a locally finitely presented Grothendieck category. An object $Y$ of $\mathcal{G}$ will be called \emph{fp-injective}  when $\text{Ext}_\mathcal{G}^1(?,Y)$  vanishes on finitely presented objects. 
\end{defi}

The following is an easy consequence of the proof of implication $1)\Longrightarrow 2)$ in \cite{Sto}[Proposition B.3], after a clear induction argument:

\begin{lema} \label{lem.fp-injectives}
Let $\mathcal{G}$ be a locally coherent Grothendieck category. If $Y$ is an fp-injective object of $\mathcal{G}$, then $\text{Ext}_\mathcal{G}^k(?;Y)$ vanishes on finitely presented objects, for all $k>0$. 
\end{lema}

Recall that if $F:\mathcal{G}\longrightarrow\text{Ab}$ is any left exact functor, then an object $Y$ of $\mathcal{G}$ is \emph{$F$-acyclic} when  the right derived functors $\mathbf{R}^kF:\mathcal{G}\longrightarrow\text{Ab}$ vanish on $Y$, for all $k>0$. Recall also that, for each $X\in\text{Ob}(\mathcal{G})$, one can  calculate $\mathbf{R}^kF(X)$ by considering $F$-acyclic resolutions. That is, if one chooses an exact sequence $0\rightarrow X\longrightarrow Y^0\stackrel{d^0}{\longrightarrow}Y^1\stackrel{d^1}{\longrightarrow}\cdots Y^n\stackrel{d^n}{\longrightarrow}\cdots $, where all the $Y^k$ are $F$-acyclic, then $\mathbf{R}^kF(X)$ is the $k$-th homology group of the complex $$\cdots 0\longrightarrow F(Y^0)\stackrel{F(d^0)}{\longrightarrow}F(Y^1)\stackrel{F(d^1)}{\longrightarrow}\cdots F(Y^n)\stackrel{F(d^n)}{\longrightarrow}\cdots , $$ for each integer $k\geq 0$. The following result seems to be well-known (see \cite[Introduction]{Gi} or \cite[Chapter 11]{P}), but we include a proof after not finding an explicit one in the literature. 

\begin{prop} \label{prop.monicExt for locally f.p.}
Let $\mathcal{G}$ be a locally finitely presented Grothendieck category, let $X$ be a finitely presented object, let   $(M_i)_{i\in I}$ be a direct system in $\mathcal{G}$ and consider the 
canonical map $\mu_k:\varinjlim\text{Ext}_\mathcal{G}^k(X,M_i)\longrightarrow\text{Ext}_\mathcal{G}^k(X,\varinjlim M_i)$, for each integer $k\geq 0$. The following assertions hold:

\begin{enumerate}
\item $\mu_0$ is an isomorphism and $\mu_1$ is a monomorphism.
\item When $\mathcal{G}$ is locally coherent, $\mu_k$ is an isomorphism, for all $k\geq 0$.
\end{enumerate}
\end{prop}
\begin{proof}

1) The case $k=0$ follows from the definition of finitely presented object.  
An element of $\varinjlim\text{Ext}_\mathcal{G}^1(X,M_i)$ is represented by a direct system $(\epsilon_i)_{i\in I}$ of exact sequences 
 $$\epsilon_i: \hspace*{0.5cm} 0\rightarrow M_i\longrightarrow N_i\longrightarrow X\rightarrow 0 $$ whose 'projection' on the first component is precisely the direct system $(M_i)_{i\in I}$ and where $X$ is viewed as a constant direct system. The image of $(\epsilon_i)$ by the canonical map $\varinjlim\text{Ext}_\mathcal{G}^1(X,M_i)\longrightarrow\text{Ext}_\mathcal{G}(X,\varinjlim M_i)$ is the induced exact sequence $$0\rightarrow\varinjlim M_i\longrightarrow\varinjlim N_i\stackrel{\pi}{\longrightarrow} X\rightarrow 0.$$ If this latter sequence splits and we fix a section $\mu :X\longrightarrow\varinjlim N_i$ for $\pi$, then, due to the fact that $X$ is a finitely presented object,  $\mu$ factors in the form $X\stackrel{\mu_j}{\longrightarrow}N_j\stackrel{u_j}{\longrightarrow}\varinjlim N_i$, for some $j\in I$, where $u_j$ is the canonical morphism to the direct limit. This immediately implies that the $j$-th sequence $\epsilon_j: \hspace*{0.5cm} 0\rightarrow M_j\longrightarrow N_j\longrightarrow X\rightarrow 0$ splits and, hence, that $(\epsilon_i)_{i\in I}$ is the zero element of $\varinjlim\text{Ext}_\mathcal{G}^1(X,M_i)$.

2) By \cite[Corollary 1.7 and subsequent remark]{AR}, we can assume without loss of generality that $I=\lambda =\{\alpha\text { ordinal: }\alpha<\lambda\}$ is an infinite limit ordinal and that, for each limit ordinal $\alpha <\lambda$, one has $M_\alpha =\varinjlim_{\beta <\alpha}M_\beta$. We now construct a direct system $(E_\alpha )_{\alpha <\lambda}$ in the category $\mathcal{C}(\mathcal{G})$ of complexes,  satisfying the following properties:

\begin{enumerate}
\item[a)] $E_\alpha :$ \hspace*{0.5cm} $\cdots 0\rightarrow E_\alpha^0\rightarrow E_\alpha^1\rightarrow \cdots \rightarrow  E_\alpha^n\rightarrow \cdots $ \hspace*{0.5cm} is a complex concentrated in degrees $\geq 0$ and $H^k(E_\alpha )=0$, for all $\alpha <\lambda$ and all $k\neq 0$;

\item[b)] $E_\alpha^n$ is an fp-injective object, for all $\alpha <\lambda$ and all integers $n\geq 0$;
\item[c)] The direct system $(H^0(E_\alpha))_{\alpha <\lambda}$ in $\mathcal{G}$ is isomorphic to $(M_\alpha )_{\alpha <\lambda}$. 

\end{enumerate}
Once the direct system $(E_\alpha )_{\alpha <\lambda}$ will be constructed, the exactness of the direct limit functor in $\mathcal{G}$ and the fact that the class of fp-injective objects is closed under taking direct limits (see \cite[Proposition B.3]{Sto}) will give that $E_\lambda:=\varinjlim_{\mathcal{C}(\mathcal{G})}E_\alpha$ is a complex of fp-injective objects concentrated in degrees $\geq 0$ whose only nonzero homology object is $H^0(E_\lambda )\cong\varinjlim_{\alpha <\lambda} M_\alpha$. That is, $E_\lambda$ is a (deleted) fp-injective resolution of $M:=\varinjlim_{\alpha <\lambda} M_\alpha$. By the previous lemma, we know that each fp-injective object is $\text{Hom}_\mathcal{G}(X,?)$-acyclic, whenever $X\in fp(\mathcal{G})$. It follows that, for such an $X$, we have that $\text{Ext}_\mathcal{G}^k(X,M)$ is the $k$-th homology abelian group of the complex $\text{Hom}_\mathcal{G}(X,E_\lambda)$.  
But, by definition of $E_\lambda$ and the fact that $\text{Hom}_\mathcal{G}(X,?)$ preserves direct limits, we have an isomorphism of complexes of abelian groups $\varinjlim_{\mathcal{C}(Ab)}(\text{Hom}_\mathcal{G}(X,E_\alpha ))\cong \text{Hom}_\mathcal{G}(X,E_\lambda)$. The $k$-th homology map will give  then the desired isomorphism $\varinjlim\text{Ext}_\mathcal{G}^k(X,M_\alpha )\stackrel{\cong}{\longrightarrow}\text{Ext}_\mathcal{G}^k(X,M)=\text{Ext}_\mathcal{G}^k(X,\varinjlim_{\alpha <\lambda}M_\alpha )$. 

It remains to construct the direct system $(E_\alpha )_{\alpha <\lambda}$ in $\mathcal{C}(\mathcal{G})$. We denote by $u_\alpha :M_\alpha\longrightarrow M_{\alpha +1}$ the morphism from the direct system $(M_\alpha )_{\alpha <\lambda}$.  For a nonlimit ordinal $\alpha$, $E_\alpha$ will be the (deleted) minimal injective resolution of $M_\alpha$. If $\alpha$ is a limit ordinal and we already have defined the direct system $(E_\beta )_{\beta <\alpha}$, then $E_\alpha =\varinjlim_{\beta <\alpha}E_\beta$, where the direct limit is taken in $\mathcal{C}(\mathcal{G})$. Note that one has that $H^0(E_\alpha )\cong\varinjlim_{\beta <\alpha} M_\beta =M_\alpha$. For the construction of $(E_\alpha)_{\alpha <\lambda}$ one just need to define the connecting chain map $E_\alpha\longrightarrow E_{\alpha +1}$, when  $\alpha<\lambda$ is any ordinal for which $E_\alpha$ is already defined. This connecting chain map is defined by chosing a family $(f_\alpha^n:E_\alpha^n\longrightarrow E_{\alpha +1}^n)_{n\geq 0}$ of morphisms in $\mathcal{G}$ such that the following diagram is commutative and the induced map $\text{Ker}(E_\alpha^0\longrightarrow E_{\alpha}^1)\longrightarrow\text{Ker}(E_{\alpha +1}^0)\longrightarrow E_{\alpha +1}^1))$ is the morphism $u_\alpha :M_\alpha\longrightarrow M_{\alpha +1}$:

$$\xymatrix{0 \ar[r] & E_{\alpha}^{0} \ar[r] \ar[d]^{f^{0}_{\alpha}} & E_{\alpha}^{1} \ar[r] \ar[d]^{f^{1}_{\alpha}} & \cdots \\ 0 \ar[r] & E_{\alpha +1 }^{0} \ar[r] & E_{\alpha +1}^{1} \ar[r] & \cdots }$$

  The reader is invited to check that the direct system $(E_\alpha )_{\alpha <\lambda}$ satisfies all the requirements. 
\end{proof}

\section{Some sufficient conditions for the heart to be a locally coherent Grothendieck category}

\begin{defi} \label{def.restricted t-structure}
Let $\mathcal{D}'$ be a full triangulated subcategory of the triangulated category $\mathcal{D}$ and let $(\mathcal{U},\mathcal{U}^\perp [1])$ be a t-structure in $\mathcal{D}$. We say that this t-structure \emph{restricts to $\mathcal{D}'$} when $(\mathcal{U}\cap\mathcal{D}',(\mathcal{U}^\perp\cap\mathcal{D}')[1])$ is a t-structure of $\mathcal{D}'$. This is equivalent to saying that, for each object $X$ of $\mathcal{D}'$, the truncation triangle $\tau_\mathcal{U}(X)\longrightarrow X\longrightarrow\tau^{\mathcal{U}^\perp}(X)\stackrel{+}{\longrightarrow}$ has its three vertices in $\mathcal{D}'$. 

\end{defi}

\begin{lema} \label{lem.restricted heart}
Let $\mathcal{D}'$ be a full triangulated subcategory of $\mathcal{D}$ and let $(\mathcal{U},\mathcal{U}^\perp [1])$ be a t-structure in $\mathcal{D}$ whose heart is $\mathcal{H}$. If the t-structure restricts to $\mathcal{D}'$, then $\mathcal{H}\cap\mathcal{D}'$ is an abelian exact subcategory of $\mathcal{H}$. 
\end{lema}
\begin{proof}
Let $f:X\longrightarrow Y$ be a morphism in $\mathcal{H}\cap\mathcal{D}'$ and complete it to a triangle, which is  in $\mathcal{D}'$:

\begin{center}
$X\stackrel{f}{\longrightarrow}Y\longrightarrow Z\stackrel{+}{\longrightarrow}$.
\end{center}
 Note that then $Z\in\mathcal{U}\cap\mathcal{U}^\perp [2]$ and, hence $Z[-1]\in\mathcal{U}^\perp [1]$. According to  \cite[Lemma 3.1]{PS1}, we have $\tilde{H}(Z)=\tau^{\mathcal{U}^\perp }(Z[-1])[1]$ and $\tilde{H}(Z[-1])=\tau_\mathcal{U}(Z[-1])$. Moreover, since the t-structure restrict to $\mathcal{D}'$ we get that both $\tilde{H}(Z)$ and $\tilde{H}(Z[-1])$ are in $\mathcal{H}\cap\mathcal{D'}$. But we then have a triangle 

\begin{center}
$\tilde{H}(Z[-1])[1]\longrightarrow Z\longrightarrow\tilde{H}(Z)\stackrel{+}{\longrightarrow}$.
\end{center}
By \cite{BBD}, we have isomorphisms $\text{Ker}_\mathcal{H}(f)\cong\tilde{H}(Z[-1])$ and  $\text{Coker}_\mathcal{H}(f)\cong\tilde{H}(Z)$ and, hence, $\mathcal{H}\cap\mathcal{D}'$ is closed under taking kernels and cokernels in $\mathcal{H}$. That it is also closed under taking finite coproducts is clear. 
\end{proof}

\begin{setting} \label{setting}
In the rest of the section we assume that $\mathcal{G}$ is a locally coherent Grothendieck category and we fix a set $\mathcal{S}$ of finitely presented generators of $\mathcal{G}$. Recall that then $\mathcal{S}$ is also a set of generators of $\mathcal{D}(\mathcal{G})$ as a triangulated category (see \cite[Lemma 9]{NSZ} or \cite[Lemma 4.10]{PV}).
\end{setting}

\begin{lema} \label{lem.from bounded to compact}
Let   $X\in\mathcal{D}^{\leq 0}(\mathcal{G})$ have bounded  finitely presented homology (i.e. $X$ is homologically bounded and  $H^k(X)\in\text{fp}(\mathcal{G})$, for all $k\in\mathbb{Z}$) and let $n$ be a natural number. There is a complex   $P\in\mathcal{C}^b(\text{sum}(\mathcal{S}))$ together with a morphism $g:P\longrightarrow X$ in $\mathcal{D}(\mathcal{G})$ such that the restriction of the natural transformation $g^*:\text{Hom}_{\mathcal{D}(\mathcal{G})}(X,?)\longrightarrow\text{Hom}_{\mathcal{D}(\mathcal{G})}(P,?)$ to $\mathcal{D}^{[-n,0]}(\mathcal{G})$ is a natural isomorphism. 
\end{lema}
\begin{proof}
By Lemma \ref{lem.bounded f.p. homology}, there is an isomorphism  $p:Q\longrightarrow X$  in $\mathcal{D}(\mathcal{G})$ such that  that $Q$ is a complex of objects in $\text{sum}(\mathcal{S})$ concentrated in degrees $\leq 0$. We have that  
$p^*:\text{Hom}_{\mathcal{D}(\mathcal{G})}(X,?)\stackrel{\cong}{\longrightarrow}\text{Hom}_{\mathcal{D}(\mathcal{G})}(Q,?)$ is a natural isomorphism of functors $\mathcal{D}(\mathcal{G})\longrightarrow\text{Ab}$. Stupid truncation at $-n-2$ gives a triangle in $\mathcal{K}(\mathcal{G})$

\begin{center}
$\sigma^{>-n-2}Q\stackrel{h}{\longrightarrow}Q\longrightarrow\sigma^{\leq -n-2}Q\stackrel{+}{\longrightarrow}$,
\end{center}
where the left vertex is in $\mathcal{C}^b(\text{sum}(\mathcal{S}))$. 
Since $\text{Hom}_{\mathcal{D}(\mathcal{G})}(\sigma^{\leq -n-2}Q[k],?)$ vanishes on $\mathcal{D}^{[-n,0]}(\mathcal{G})$, for $k=-1,0$, we get that the restriction of the natural transformation $h^*:\text{Hom}_{\mathcal{D}(\mathcal{G})}(Q,?)\longrightarrow\text{Hom}_{\mathcal{D}(\mathcal{G})}(\sigma^{>-n-2}Q,?)$ to $\mathcal{D}^{[-n,0]}(\mathcal{G})$ is an isomorphism. 
Putting $P:=\sigma^{>-n-2}Q$,   the desired morphism $g$ is the composition $P\stackrel{h}{\longrightarrow}Q\stackrel{p}{\longrightarrow}X$. 
\end{proof}

\begin{rem} \label{rem.tildeH on aisle and co-aisle}
Let $(\mathcal{U},\mathcal{U}^\perp [1])$ be a t-structure in any triangulated category $\mathcal{D}$ and suppose that it restricts to a full triangulated subcategory $\mathcal{D}'$. If $\tilde{H}:\mathcal{D}\longrightarrow\mathcal{H}$ is the associated cohomological functor, then $\tilde{H}(M)$ is in $\mathcal{H}\cap\mathcal{D}'$, for all $M\in\mathcal{D}'$. This is due to the fact that we have $\tau_\mathcal{U}(\mathcal{D}')\subseteq\mathcal{D}'$ and $\tau^{\mathcal{U}^\perp [1]}(\mathcal{D}')\subseteq\mathcal{D}'$.
  
\end{rem}

The following technical result is crucial for the main results of the paper.

\begin{prop} \label{prop.locally coherent Grothendieck heart}
Let $\mathcal{G}$ and $\mathcal{S}$ be as in Setting \ref{setting}, 
 let $(\mathcal{U},\mathcal{U}^\perp [1])$ be a t-structure in $\mathcal{D}(\mathcal{G})$, with heart $\mathcal{H}$,  and let $\tilde{H}:\mathcal{D}(\mathcal{G})\longrightarrow\mathcal{H}$ be the associated cohomological functor. Suppose that the following conditions hold:

\begin{enumerate}
\item $(\mathcal{U},\mathcal{U}^\perp [1])$ restricts to $\mathcal{D}^b(\text{fp}(\mathcal{G}))$;
\item There exist integers $m\leq n$ such that $\mathcal{D}^{\leq m}(\mathcal{G})\subseteq\mathcal{U}\subseteq\mathcal{D}^{\leq n}(\mathcal{G})$;
\item $\mathcal{H}\cap\mathcal{D}^b(\text{fp}(\mathcal{G}))$ is a (skeletally small) class of generators of $\mathcal{H}$;
\item For each direct system $(M_i)_{i\in I}$ in $\mathcal{H}$, for each $S\in\mathcal{S}$ and for each $k\in\mathbb{Z}$, the canonical map $\eta_{S[k]}:\varinjlim\text{Hom}_{\mathcal{D}(\mathcal{G})}(S[k],M_i)\longrightarrow\text{Hom}_{\mathcal{D}(\mathcal{G})}(S[k],\varinjlim_\mathcal{H}M_i)$ is an isomorphism;
\end{enumerate}
Then $\mathcal{H}$ is a locally coherent Grothendieck category on which $\mathcal{H}\cap\mathcal{D}^b(\mathcal{G})$ is the class of its finitely presented objects. 
\end{prop}
\begin{proof}
Consider the cohomological functor $H':=\coprod_{S\in\mathcal{S}}\text{Hom}_{\mathcal{D}(\mathcal{G})}(S,?):\mathcal{D}(\mathcal{G})\longrightarrow Ab$. Using condition 4 and the fact that $\mathcal{S}$ is a set of  generators of $\mathcal{D}(\mathcal{G})$,  we see that, with the terminology of \cite[Section 3]{PS1},  the pair $(H',+\infty )$ is a cohomological datum in $\mathcal{D}(\mathcal{G})$ for $\mathcal{H}$.  Then \cite[Proposition 3.4]{PS1} says  that $\mathcal{H}$ is an AB5 abelian category. But condition 3 says that it has a set of generators, so that $\mathcal{H}$ is a Grothendieck category. 

 Fix a direct system $(M_i)_{i\in I}$ in $\mathcal{H}$ in the sequel and consider   the full subcategory $\mathcal{C}$ of $\mathcal{D}(\mathcal{G})$ consisting of those complexes $X$ such that $$\eta_{X[k]}:\varinjlim\text{Hom}_{\mathcal{D}(\mathcal{G})}(X[k],M_i)\longrightarrow\text{Hom}_{\mathcal{D}(\mathcal{G})}(X[k],\varinjlim_\mathcal{H}M_i)$$ is an isomorphism, for all $k\in\mathbb{Z}$. Using $5$-Lemma, one readily sees that $\mathcal{C}$ is a thick subcategory of $\mathcal{D}(\mathcal{G})$ which, by condition 4, contains $\mathcal{S}$. We then have $\text{thick}_{\mathcal{D}(\mathcal{G})}(\mathcal{S})\subseteq\mathcal{C}$. In particular, if a complex $X\in\mathcal{C}^b(\text{sum}(\mathcal{S}))$ is viewed as an object of $\mathcal{D}(\mathcal{G})$, then $X\in\mathcal{C}$. 

We now claim that $\eta_X$ is also an isomorphism, for each $X\in\mathcal{D}^b(fp(\mathcal{G}))$. Indeed,   
condition 2 implies that $\mathcal{H}\subseteq\mathcal{D}^{[m,n]}(\mathcal{G})$. Let $X\in\mathcal{D}^b(fp(\mathcal{G}))$ be arbitrary. Replacing $n$ by a larger integer if necessary, we can assume that $X\in\mathcal{D}^{\leq n}(\mathcal{G})$. Then the obvious generalization of Lemma \ref{lem.from bounded to compact} says that there exist a  $P\in\mathcal{C}^b(\text{sum}(\mathcal{S}))$ and a morphism $g:P\longrightarrow X$ in $\mathcal{D}(\mathcal{G})$ such that the natural transformation $g^*:\text{Hom}_{\mathcal{D}(\mathcal{G})}(X,?)\longrightarrow\text{Hom}_{\mathcal{D}(\mathcal{G})}(P,?)$ is an isomorphism when evaluated on objects of $\mathcal{D}^{[m,n]}(\mathcal{G})$. We then have the following commutative diagram
$$\xymatrix{ \varinjlim{ \text{Hom}_{\mathcal{D}(\mathcal{G})}(X,M_i)} \ar[r]^{\eta_X} \ar[d]^{g^{*}}& \text{Hom}_{\mathcal{D}(\mathcal{G})}(X, \varinjlim_{\mathcal{H}}{M_i}) \ar[d]^{g^{*}} \\ \varinjlim{ \text{Hom}_{\mathcal{D}(\mathcal{G})}(P,M_i)} \ar[r]^{\eta_P} & \text{Hom}_{\mathcal{D}(\mathcal{G})}(P, \varinjlim_{\mathcal{H}}{M_i})  }$$
where the vertical arrows are isomorphisms and, due to the previous paragraph, also the lower horizontal arrow is an isomorphism. This settles our claim. In particular, it implies that $\mathcal{H}\cap\mathcal{D}^b(\mathcal{G})$ is a class of finitely presented objects in $\mathcal{H}$ and, by conditions 1 and 3, it is a class of generators of $\mathcal{H}$ (see Remark \ref{rem.tildeH on aisle and co-aisle}). In particular $\mathcal{H}$ is locally finitely presented. Note also that, by condition 1 and  Lemma \ref{lem.restricted heart}, we know that $\mathcal{H}\cap\mathcal{D}^b(\text{fp}(\mathcal{G}))$ is closed under taking cokernels (and kernels) in $\mathcal{H}$. It immediately follows that each finitely presented object of $\mathcal{H}$ is in $\mathcal{H}\cap\mathcal{D}^b(\text{fp}(\mathcal{G}))$ since it is the cokernel of a morphism in this latter category. Then we have that $\mathcal{H}\cap\mathcal{D}^b(\text{fp}(\mathcal{G}))=\text{fp}(\mathcal{H})$, and this is an abelian exact subcategory of $\mathcal{H}$.  Therefore $\mathcal{H}$ is locally coherent.

\end{proof}

\begin{rem}
Condition 1 of last proposition is not necessary for the heart to be a locally coherent Grothendieck category. Indeed, by \cite[Corollary 5.12]{PS3} and using the terminology of that reference, if $R$ is a commutative Noetherian ring and $Z\subsetneq\text{Spec}(R)$ is a perfect sp-subset, then $(\mathcal{U},\mathcal{U}^\perp [1])$ is a t-structure whose heart is equivalent to $R_Z-\text{Mod}$, where $\mathcal{U}$ consists of the complexes $U$ such that $\text{Supp}(H^j(U))\subseteq Z$, for all $j>-1$. Then the heart is locally coherent since $R_Z$ is a noetherian commutative ring. But the associated sp-filtration $\phi =\phi_\mathcal{U}$ of $\text{Spec}(R)$ (see \cite[Section 2.8 and Theorem 3.11]{AJS}) is given by $\phi (i)=\text{Spec}(R)$, for $i\leq -1$, and $\phi (i)=Z$, for all $i>-1$. This sp-filtration does not satisfy in general the weak Cousin condition, in whose case $(\mathcal{U},\mathcal{U}^\perp [1])$ does not restrict to $\mathcal{D}^b(\text{fp}(R-\text{Mod}))\cong\mathcal{D}_{fg}^b(R)$ (see \cite[Corollary 4.5]{AJS}). As an example of last situation, consider $R=\mathbb{Z}$ and $Z=\text{Spec}(\mathbb{Z})\setminus\{0\}$, so that $R_Z=\mathbb{Q}$. We have a canonical triangle $\mathbb{Q}/\mathbb{Z}[-1]\longrightarrow\mathbb{Z}\longrightarrow\mathbb{Q}\stackrel{+}{\longrightarrow}$,  where $\mathbb{Q}/\mathbb{Z}[-1]\in\mathcal{U}$ and $Q\in\mathcal{U}^\perp$. 
\end{rem}

\section{The case of the Happel-Reiten-Smal\o \hspace*{0.1cm}   t-structure}

Recall (see \cite{HRS}) that if $\mathcal{A}$ is any abelian category and $\mathbf{t}=(\mathcal{T},\mathcal{F})$ is a torsion pair in $\mathcal{A}$, then $(\mathcal{U}_\mathbf{t},\mathcal{U}_\mathbf{t}^\perp [1])=(\mathcal{U}_\mathbf{t},\mathcal{V}_\mathbf{t})$ is a t-structure in $\mathcal{D}(\mathcal{A})$, where 

\begin{center}
$\mathcal{U}_\mathbf{t}=\{U\in\mathcal{D}^{\leq 0}(\mathcal{A})\text{: }H^0(U)\in\mathcal{T}\}$

and

$\mathcal{V}_\mathbf{t}=\{V\in\mathcal{D}^{\geq -1}(\mathcal{A})\text{: }H^{-1}(V)\in\mathcal{F}\}$.
\end{center}

This t-structure will be called the \emph{Happel-Reiten-Smal\o \hspace*{0.1cm}(or just HRS) t-structure associated to $\mathbf{t}$}. In this paper we are only interested in the case when $\mathcal{A}=\mathcal{G}$ is a locally coherent Grothendieck category.  

Therefore, all throughout this section, $\mathcal{G}$ will be a locally coherent Grothendieck category and $\mathbf{t}=(\mathcal{T},\mathcal{F})$ will be a torsion pair in $\mathcal{G}$. Recall that $\mathbf{t}$ is said to be of \emph{finite type} when the torsion radical $t:\mathcal{G}\longrightarrow\mathcal{T}$ preserves direct limits or, equivalently, when $\mathcal{F}$ is closed under taking direct limits in $\mathcal{G}$ (see \cite[Section 2]{Kr}). We shall say that $\mathbf{t}$ \emph{restricts to $fp(\mathcal{G})$} when $t(X)$ is in $fp(\mathcal{G})$, for each $X\in fp(\mathcal{G})$. Note that this is equivalent to saying that $\mathbf{t}'=(\mathcal{T}\cap fp(\mathcal{G}),\mathcal{F}\cap fp(\mathcal{G}))$ is a torsion pair in $fp(\mathcal{G})$. 

\begin{prop} \label{prop.restriction of HRS t-structure}
Let $(\mathcal{U}_\mathbf{t},\mathcal{U}_\mathbf{t}^\perp [1])$ be the HRS t-structure in $\mathcal{D}(\mathcal{G})$ associated to  $\mathbf{t}$. The following assertions are equivalent:

\begin{enumerate}
\item The t-structure $(\mathcal{U_\mathbf{t}},\mathcal{U}_\mathbf{t}^\perp [1])$ restricts to $\mathcal{D}^b(\text{fp}(\mathcal{G}))$;
\item The torsion pair $\mathbf{t}$ restricts to $fp(\mathcal{G})$. 
\end{enumerate}
In particular, if $\mathcal{G}$ is locally Noetherian  then $(\mathcal{U_\mathbf{t}},\mathcal{U}_\mathbf{t}^\perp [1])$  restricts to $\mathcal{D}^b(\text{fp}(\mathcal{G}))$. 
\end{prop}
\begin{proof}
Given $M\in\mathcal{D}^b(\text{fp}(\mathcal{G}))$, we have canonical triangles in $\mathcal{D}(\mathcal{G})$

\begin{center}

$\tau^{\leq -1}M\longrightarrow M\longrightarrow\tau^{\geq 0}M\stackrel{+}{\longrightarrow}$

$t(H^0(M))[0]\longrightarrow\tau^{\geq 0}M\longrightarrow W\stackrel{+}{\longrightarrow}$,
\end{center}
where $W\in\mathcal{D}^{\geq 0}(\mathcal{G})$,  $H^0(W)\cong\frac{H^0(M)}{t(H^0(M))}$ and $H^k(W)=H^k(M)$, for all $k>0$. Then $W\in\mathcal{U}_\mathbf{t}^\perp =\mathcal{V}_\mathbf{t}[-1]$. Applying the octhaedrom axiom to the last two triangles, we  obtain two new triangles

\begin{center}
$\tau^{\leq -1}M\longrightarrow U\longrightarrow t(H^0(M))[0]\stackrel{+}{\longrightarrow}$

$U\longrightarrow M\longrightarrow W\stackrel{+}{\longrightarrow}$.
\end{center}
It follows from the first triangle that $U\in\mathcal{U}_\mathbf{t}$ since the outer vertices of the triangle are in $\mathcal{U}_\mathbf{t}$. We then conclude that the second triangle is precisely the truncation triangle of $M$ with respect to $(\mathcal{U}_\mathbf{t},\mathcal{U}_\mathbf{t}^\perp [1])$. 

The last truncation triangle is in $\mathcal{D}^b(\text{fp}(\mathcal{G}))$ if, and only if, $U\in\mathcal{D}^b(\text{fp}(\mathcal{G}))$. But this happens exactly when $t(H^0(M))[0]\in\mathcal{D}^b(\text{fp}(\mathcal{G}))$. That is, exactly when $t(H^0(M))$ is a finitely presented object. 
The equivalence of assertions 1 and 2 is now clear. 

Noting that $\mathcal{G}$ is locally coherent all throughout this section, when $\mathcal{G}$ is also locally noetherian we have that $\text{fp}(\mathcal{G})$ coincides with the class $\text{noeth}(\mathcal{G})$ of noetherian objects, which is obviously closed under taking subobjects. Therefore $\mathbf{t}$ always restricts to $\text{fp}(\mathcal{G})$. 
\end{proof}

We are now ready to prove the first main result of the paper.

\begin{teor} \label{teor.HRS}
Let $\mathcal{G}$ be a locally coherent Grothendieck category, let $\mathbf{t}=(\mathcal{T},\mathcal{F})$ be a torsion pair in $\mathcal{G}$, let $(\mathcal{U}_\mathbf{t},\mathcal{U}_\mathbf{t}^\perp [1])$ be the associated t-structure in $\mathcal{D}(\mathcal{G})$ and let $\mathcal{H}_\mathbf{t}$ be its heart. The following assertions are equivalent:

\begin{enumerate}
\item $(\mathcal{U}_\mathbf{t},\mathcal{U}_\mathbf{t}^\perp [1])$ restricts to $\mathcal{D}^b(\text{fp}(\mathcal{G}))$ and $\mathcal{H}_\mathbf{t}$ is a locally coherent Grothendieck category (with $\mathcal{H}_\mathbf{t}\cap\mathcal{D}^b(\text{fp}(\mathcal{G})$ as the class of finitely presented objects). 
\item $\mathbf{t}$ is of finite type and restricts to $fp(\mathcal{G})$.
\item There exists a torsion pair $\mathbf{t}'=(\mathcal{T}',\mathcal{F}')$ in $fp(\mathcal{G})$ such that $\mathbf{t}=(\varinjlim\mathcal{T}',\varinjlim\mathcal{F}')$. 

When in addition $\mathcal{G}$ is locally Noetherian, these assertions are also equivalent to:

\item $\mathbf{t}$ is of finite type.
\end{enumerate}
\end{teor}
\begin{proof}
All throughout the proof, we fix a set $\mathcal{S}$ of finitely presented generators of $\mathcal{G}$. 

$1)\Longrightarrow 2)$ By Proposition \ref{prop.restriction of HRS t-structure}, we know that $\mathbf{t}$ restricts to $fp(\mathcal{G})$ and, by \cite{PS1}[Theorem 4.8], we know that $\mathbf{t}$ is of finite type.

$2)\Longrightarrow 3)$ If we put $\mathcal{T}'=\mathcal{T}\cap\text{fp}(\mathcal{G})$ and $\mathcal{F}'=\mathcal{F}\cap\text{fp}(\mathcal{G})$, then $\mathbf{t}'=(\mathcal{T}',\mathcal{F}')$ is a torsion pair in $\text{fp}(\mathcal{G})$ since $\mathbf{t}$ restricts to $\text{fp}(\mathcal{G})$. By \cite{CB}[Lemma 4.4], we know that $(\varinjlim\mathcal{T}',\varinjlim\mathcal{F}')$ is a torsion pair in $\mathcal{G}$. But $\mathcal{T}$ and $\mathcal{F}$ are closed under taking direct limits in $\mathcal{G}$, which implies that $\varinjlim \mathcal{T}'\subseteq\mathcal{T}$ and $\varinjlim\mathcal{F}'\subseteq\mathcal{F}$. Since we always have $\mathcal{F}=\mathcal{T}^\perp\subseteq(\varinjlim\mathcal{T}')^\perp=\varinjlim\mathcal{F}'$ we conclude that $(\mathcal{T},\mathcal{F})=(\varinjlim\mathcal{T}',\varinjlim\mathcal{F}')$. 

$3)\Longrightarrow 2)$ is clear.

$2)\Longrightarrow 1)$ The finite type condition of $\mathbf{t}$ implies that $\mathcal{H}_\mathbf{t}$ is a Grothendieck category (see \cite[Theorem 1.2]{PS2}). Let now $(M_i)_{i\in I}$ be a direct system in $\mathcal{H}_\mathbf{t}$. Bearing in mind that $\mathcal{F}$ is closed under taking direct limits in $\mathcal{G}$ and using \cite[Proposition 4.2]{PS1}, we get an exact sequence in $\mathcal{H}_\mathbf{t}$: 

\begin{center}
$0\rightarrow (\varinjlim H^{-1}(M_i))[1]\longrightarrow\varinjlim_{\mathcal{H}_\mathbf{t}}M_i\longrightarrow (\varinjlim H^0(M_i))[0]\rightarrow 0.$
\end{center}

To abbreviate, let us put $(X,Y)=\text{Hom}_{\mathcal{D}(\mathcal{G})}(X,Y)$, for all $X,Y\in\mathcal{D}(\mathcal{G})$. Then, for each $S\in\mathcal{S}$ and each $k\in\mathbb{Z}$, we have a commutative diagram of abelian groups with exact columns, where the horizontal arrows are the canonical morphisms:


$$\xymatrix @d {\varinjlim (S[k],H^0(M_i)[-1])\ar[r] \ar[d] & \varinjlim (S[k],H^{-1}(M_i)[1])\ar[r] \ar[d] & \varinjlim (S[k],M_i) \ar[r] \ar[d] & \varinjlim (S[k],H^0(M_i)[0]) \ar[d] \ar[r] & \varinjlim(S[k],H^{-1}(M_i)[2]) \ar[d] \\ (S[k],(\varinjlim H^{0}(M_i))[-1])\ar[r] & (S[k],(\varinjlim H^{-1}(M_i))[1]) \ar[r] & (S[k],\varinjlim_{\mathcal{H}_\mathbf{t}}M_i) \ar[r] & (S[k],(\varinjlim H^{0}(M_i))[0]) \ar[r] & (S[k],(\varinjlim H^{-1}(M_i))[2])  }$$

By Proposition \ref{prop.monicExt for locally f.p.}, we have that the two upper most and the two lower most horizontal  arrows are isomorphisms, which implies that also the canonical map $\varinjlim (S[k],M_i)\longrightarrow (S[k],\varinjlim_{\mathcal{H}_\mathbf{t}}M_i)$ is an isomorphism. 

 We will check now that all conditions 1-4 of Proposition \ref{prop.locally coherent Grothendieck heart} are satisfied by $(\mathcal{U}_\mathbf{t},\mathcal{U}_\mathbf{t}[1])$.   By Proposition \ref{prop.restriction of HRS t-structure}, we know that $(\mathcal{U}_\mathbf{t},\mathcal{U}_\mathbf{t}^\perp [1])$ restricts to $\mathcal{D}^b(fp(\mathcal{G}))$ and, by definition of the HRS t-structure, we know that $\mathcal{D}^{\leq -1}(\mathcal{G})\subseteq\mathcal{U}_\mathbf{t}\subseteq\mathcal{D}^{\leq 0}(\mathcal{G})$, so that conditions 1 and 2 of the mentioned proposition hold. Moreover, the previous paragraph says that also condition 4  holds.

We will finally check that each object of $\mathcal{H}_\mathbf{t}$ is an epimorphic image of a coproduct of objects of $\mathcal{H}_\mathbf{t}\cap\mathcal{D}^b(\text{fp}(\mathcal{G}))$, which will give condition 3 of Proposition \ref{prop.restriction of HRS t-structure} and will end the proof. Let $M$ be any object of $\mathcal{H}_\mathbf{t}$ and let us write $H^0(M)=\varinjlim T_i$, for some direct system $(T_i)_{i\in I}$ in $\mathcal{T}\cap fp(\mathcal{G})$. Note that this is possible since $\mathcal{T}=\varinjlim (\mathcal{T}\cap fp(\mathcal{G}))$. Considering the canonical exact sequence $0\rightarrow H^{-1}(M)[1]\longrightarrow M\longrightarrow H^0(M)[0]\rightarrow 0$ and pulling it back, for each $i\in I$,  along the obvious map $T_i[0]\longrightarrow H^0(M)[0]$, we get a direct system of exact sequences in $\mathcal{H}_\mathbf{t}$ 
$$0\rightarrow H^{-1}(M)[1]\longrightarrow M_i\longrightarrow T_i[0]\rightarrow 0. $$ Since $\mathcal{H}_{\mathbf{t}}$ is a Grothendieck category it immediately follows that $M=\varinjlim_{\mathcal{H}_\mathbf{t}}M_i$, so that $M$ is an epimorphic image of $\coprod_{i\in I}M_i$. Replacing $M$ by any of the $M_i$, we can and shall assume in the rest of the proof that $H^0(M)\in\mathcal{T}\cap fp(\mathcal{G})$. We then write $M$ as a complex $\cdots 0\longrightarrow M^{-1}\longrightarrow M^0\longrightarrow 0 \cdots $ concentrated in degrees $-1$ and $0$. Note that if we put $M^0=\varinjlim M^0_i$, where $(M^0_i)_{i\in I}$ is a direct system in $fp (\mathcal{G})$, then some composition $M^0_j\stackrel{\iota_j}{\longrightarrow}\varinjlim M_i^0=M^0\stackrel{p}{\twoheadrightarrow}H^0(M)$ should be an epimorphism, because $H^0(M)$ is finitely presented. Replacing $M^0$ by $M_j^0$ if necessary, we can assume in the sequel that $M^0$ is also finitely presented. 

Once we assume that $H^0(M)$ and $M^0$ are both finitely presented, we follow the lines of the proof of \cite[Proposition 4.7]{PS1} with an easy adaptation. The details are left to the reader.  
Since $M^{-1}$ is a direct limit of finitely presented objects, we can fix an epimorphism $\coprod_{j\in J}X_j\twoheadrightarrow M^{-1}$ in $\mathcal{G}$, where $X_j\in fp(\mathcal{G})$ for all $j\in J$. Now we construct a 4-row commutative diagram as in the mentioned proof, where $G^{(J)}$ and $G^{(F)}$ are replaced in our case by $\coprod_{j\in J}X_j$ and $\coprod_{j\in F}X_j$, respectively. The key point now is that the appearing $U_F$ and $X_F$ are finitely presented objects. Since $\mathbf{t}$ restricts to $fp(\mathcal{G})$, we also know that $t(X_F)$ (and also $M_F^0$) is finitely presented, for each finite subset $F\subseteq J$. If now $L=\tilde{H}_{| \mathcal{U}_\mathbf{t}}:\mathcal{U}_\mathbf{t}\longrightarrow\mathcal{H}_\mathbf{t}$ is the left adjoint to inclusion functor (see \cite[Lemma 3.1]{PS1}), the mentioned proof shows that we have epimorphisms $\coprod_{F\subset J,F\text{ }finite}L(K_F)\twoheadrightarrow L(K_J)$ and $L(K_J)\twoheadrightarrow M$ in $\mathcal{H}_\mathbf{t}$, where $L(K_F)$ is the object of $\mathcal{H}_\mathbf{t}$ represented by the complex $\cdots 0\longrightarrow\frac{\coprod_{j\in F}X_j}{t(U_F)}\longrightarrow M^0_F\longrightarrow 0 \cdots $, concentrated in degrees $-1$ and $0$. But $t(U_F)$ is finitely presented, because so is $U_F$. It follows that the latter complex is a complex of finitely presented objects, and hence $L(K_F)\in\mathcal{H}_\mathbf{t}\cap\mathcal{D}^b(\text{fp}(\mathcal{G}))$. 

$4)\Longrightarrow 2)=3)$ If $\mathcal{G}$ is locally Noetherian, each torsion pair restricts to its subcategory of noetherian objects, that is, to $\text{fp}(\mathcal{G})$. 
\end{proof}

\section{The heart of a restricted t-structure in the derived category of a commutative noetherian ring}

All throughout this section $R$ is a commutative noetherian ring. To apply the results of earlier sections, we will consider $\mathcal{G}=R-\text{Mod}$ the category of all $R$-modules, which is a locally noetherian Grothendieck category.  Then we have that $fp(\mathcal{G})=R-\text{mod}$ is the subcategory of finitely generated $R$-modules and, as usual (see comments preceding Definition \ref{def.fp-injective}), we identify $\mathcal{D}^b_{fg}(R):=\mathcal{D}^b_{fp}(R-\text{Mod})$ with $\mathcal{D}^b(R-\text{mod})$. 

Recall that a \emph{filtration by supports} or \emph{sp-filtration} of $\text{Spec}(R)$ is a decreasing map $\phi:\mathbb{Z}\longrightarrow\mathcal{P}(\text{Spec}(R))$ such that $\phi(i)\subseteq \text{Spec}(R)$ is a stable under specialization subset,  for each $i\in \mathbb{Z}$. Filtrations by supports turn out to be in bijection with the compactly generated t-structures in $\mathcal{D}(R)$ (see \cite[Theorem 3.11]{AJS}). Concretely, given an sp-filtration $\phi$ and putting $\mathcal{U}_\phi =\{U\in\mathcal{D}(R)\text{: }\text{Supp}(H^i(U))\subseteq\phi (i)\text{, for all }i\in\mathbb{Z}\}$, we get a compactly generated t-structure $\tau_\phi =(\mathcal{U}_\phi ,\mathcal{U}_\phi^\perp [1])$ and the assignment $\phi\rightsquigarrow\tau_\phi$ gives the mentioned bijection. All through this section, the reader is referred to \cite{AJS} for all non-defined terms that we might use.

\begin{lema} \label{lem.Hom under localization}
Let $X\in\mathcal{D}^b_{fg}(R)$ and $Y\in\mathcal{D}^+(R)$. For each $\mathbf{p}\in\text{Spec}(R)$, the canonical map

\begin{center}
$\text{Hom}_{\mathcal{D}(R)}(X,Y)_\mathbf{p}\longrightarrow\text{Hom}_{\mathcal{D}(R_\mathbf{p})}(X_\mathbf{p},Y_\mathbf{p})$ 
\end{center}
is an isomorphism. 
\end{lema}
\begin{proof}
Let us fix $Y\in\mathcal{D}^+(R)$, which we consider to be a bounded below complex of injective $R$-modules.  For each $Z$ in $\mathcal{D}^b_{fg}(R)$, we denote by $\eta_Z$ the canonical map $\text{Hom}_{\mathcal{D}(R)}(Z,Y)_\mathbf{p}\longrightarrow\text{Hom}_{\mathcal{D}(R_\mathbf{p})}(Z_\mathbf{p},Y_\mathbf{p})$. We then consider the full subcategory $\mathcal{C}$ of $\mathcal{D}^b_{fg}(R)$ consisting of those $Z$ such that $\eta_{Z[k]}$ is an isomorphism, for all $k\in\mathbb{Z}$. It is clear that  $\mathcal{C}$ is a thick subcategory of $\mathcal{D}^b_{fg}(R)$. 

We claim that $M[0]\in\mathcal{C}$, for each finitely generated $R$-module $M$. Once this is proved, the proof will be finished. Indeed, we will conclude that $\mathcal{C}=\mathcal{D}^b_{fg}(R)$ since each $Z\in\mathcal{D}^b_{fg}(R)$ is  a finite iterated extension of the stalk complexes $H^{-k}(Z)[k]$, and each $H^{-k}(Z)$ is finitely generated.  Recall that $\text{Hom}_{\mathcal{D}(R)}(M[-k],Y)$ is the $k$-th homology module of the complex of $R$-modules $\text{Hom}_R(M,Y)$. Similarly, $\text{Hom}_{\mathcal{D}(R_\mathbf{p})}(M_\mathbf{p}[-k],Y_\mathbf{p})$ is the $k$-th homology module of the complex of $R_\mathbf{p}$-modules $\text{Hom}_{R_\mathbf{p}}(M_\mathbf{p},Y_\mathbf{p})$ since $Y_\mathbf{p}$ is a bounded below complex of injective $R_\mathbf{p}$-modules. The claim follows from the exactness of the localization at $\mathbf{p}$ and from the truth of the result when $Y$ is a module (see, e.g., \cite{Ku}[Proposition IV.1.10]). 
\end{proof}

\begin{lema} \label{lem.L(U)}
Let $R$ be connected, let $(\mathcal{U},\mathcal{U}^\perp [1])$ be a compactly generated t-structure in $\mathcal{D}(R)$ which restricts to $\mathcal{D}^b_{fg}(R)$, let $\mathcal{H}$ be its heart and let $U\in\mathcal{D}^{-}(R)\cap\mathcal{U}$ be a complex with finitely generated homology modules. Then $\tilde{H}(U)$ is in $\mathcal{H}\cap\mathcal{D}^b_{fg}(R)$.
\end{lema}
\begin{proof}
 Let $\phi$ be the sp-filtration of $\text{Spec}(R)$ associated to $(\mathcal{U},\mathcal{U}^\perp [1])$. By \cite[Corollaries 4.5 and 4.8]{AJS}, we know that there exists some $j_0\in\mathbb{Z}$ such that $\phi (j_0)=\text{Spec}(R)$. Without loss of generality, we assume that $j_0=0$. We then have $\mathcal{D}^{\leq 0}(R)\subseteq\mathcal{U}$ and $\mathcal{H}\subseteq\mathcal{D}^{\geq 0}(R)$. By considering now for the object $U$ of the statement the canonical truncation triangle

\begin{center}
$\tau^{\leq 0}(U[-1])\longrightarrow U[-1]\stackrel{g}{\longrightarrow}\tau^{>0}(U[-1])\stackrel{+}{\longrightarrow}$
\end{center}
and applying the octhaedrom axiom, we see that $\tau^{\mathcal{U}^\perp}(g):\tau^{\mathcal{U}^\perp}(U[-1])\longrightarrow\tau^{\mathcal{U}^\perp}(\tau^{>0}(U[-1]))$ is an isomorphism. But the codomain of this morphism is in $\mathcal{D}^b_{fg}(R)$ since $\tau^{>0}(U[-1])\in\mathcal{D}^b_{fg}(R)$ and the t-structure $(\mathcal{U},\mathcal{U}^\perp [1])$ restricts to $\mathcal{D}^b_{fg}(R)$. Then $\tilde{H}(U)=\tau^{\mathcal{U}^\perp}(U[-1])[1]$ is in $\mathcal{D}^b_{fg}(R)$ (see \cite[Lemma 3.1]{PS1}). 
\end{proof}

We are now ready to prove the main result of the paper.

\begin{teor} \label{teor.main}
Let $R$ be a commutative Noetherian ring and let $(\mathcal{U},\mathcal{U}^\perp [1])$ be a compactly generated t-structure in $\mathcal{D}(R)$ which restricts to $\mathcal{D}_{fg}^b(R)$. The heart $\mathcal{H}$ of this t-structure is a locally coherent Grothendieck category where  $\mathcal{H}\cap\mathcal{D}_{fg}^b(R)$ is the subcategory of its finitely presented  objects. 
\end{teor}
\begin{proof}
All throughout the proof, without loss of generality, we assume that $R$ is connected. Remember that then the associated sp-filtration $\phi$ satisfies the weak Cousin condition and, hence, has the property that $\phi (i)=\text{Spec}(R)$, for $i\ll 0$ (see \cite{AJS}[Theorem 4.4 and Corollary 4.8]). This in turn implies that $\mathcal{H}=\mathcal{H}_\phi\subseteq\mathcal{D}^{\geq m}(R)$, for some $m\in\mathbb{Z}$. Moreover, by \cite[Theorem 4.10]{PS3}, we know that $\mathcal{H}=\mathcal{H}_\phi$ is a Grothendieck category.  

{\it Step 1: $\mathcal{H}\cap\mathcal{D}^b_{fg}(R)$ is a (skeletally small) class of generators of $\mathcal{H}$:}  Let $\mathcal{U}'$ denote the full subcategory of $\mathcal{U}$ consisting of complexes in $\mathcal{U}\cap\mathcal{D}^-(R)$ which have finitely generated homology modules. Each object of $\mathcal{U}'$ is isomorphic in $\mathcal{D}(R)$ to a bounded above complex of finitely generated $R$-modules. Let $L=\tilde{H}_{|\mathcal{U}}:\mathcal{U}\longrightarrow\mathcal{H}$ be the left adjoint to the inclusion functor $\mathcal{H}\hookrightarrow\mathcal{U}$. A slight modification of the proof of \cite[Proposition 3.10]{PS3} shows that $\mathcal{X}:=L(\mathcal{U}')$ is a skeletally small class of generators of $\mathcal{H}$. By lemma \ref{lem.L(U)}, we get that $\mathcal{X}\subseteq\mathcal{H}\cap\mathcal{D}_{fg}^b(R)$, which ends this first step. 

{\it Step 2: the result is true when $\phi$ is eventually trivial (i.e. when $\phi (i)=\emptyset$, for some $i\in\mathbb{Z}$):}  We shall check all conditions 1-4 of Proposition \ref{prop.locally coherent Grothendieck heart}. Without loss of generality, we assume that the filtration is

$$\text{Spec}(R)=\cdots \phi (-n-1)=\phi (-n)\supsetneq\phi (-n+1)\supseteq \cdots \supseteq\phi (0)\supsetneq\phi (1)=\phi (2)=\cdots =\emptyset, \hspace*{1cm} (*) $$
in which case we have that $\mathcal{D}^{\leq -n}(R)\subseteq\mathcal{U}\subseteq\mathcal{D}^{\leq 0}(R)$ and $\mathcal{H}=\mathcal{H}_\phi\subseteq\mathcal{D}^{[-n,0]}(R)$ (see \cite[Lemma 4.1]{PS3}). Then condition 2 of Proposition \ref{prop.locally coherent Grothendieck heart} holds and condition 1 holds by hypothesis. Moreover, step 1 of this proof gives condition 3 of that proposition.
Finally, bearing in mind that we have a natural isomorphism $H^k\cong\text{Hom}_{\mathcal{D}(R)}(R[-k],?)$ of functors $\mathcal{D}(R)\longrightarrow R-\text{Mod}$,  by taking $\mathcal{S}=\{R\}$ and using \cite[Theorem 4.9]{PS3} we also get that  condition 4 holds. 

{\it Step 3: The general case}. The proof reduces to check that $\mathcal{H}\cap\mathcal{D}^b_{fg}(R)\subseteq fp(\mathcal{H})$. Indeed, if this is proved, then step 1 implies that $\mathcal{H}$ is locally finitely presented  and that each object in $fp(\mathcal{H})$ is the cokernel of a morphism in $\mathcal{H}\cap\mathcal{D}^b_{fg}(R)$. It will follow from 
 Lemma \ref{lem.restricted heart} that  $fp(\mathcal{H})=\mathcal{H}\cap\mathcal{D}^b_{fg}(R)$ and that this is an abelian exact subcategory of $\mathcal{H}$. That is, $\mathcal{H}$ will be a locally coherent Grothendieck category with $\mathcal{H}\cap\mathcal{D}^b_{fg}(R)$ as its class of finitely presented objects. 

We then prove the inclusion $\mathcal{H}\cap\mathcal{D}^b_{fg}(R)\subseteq fp(\mathcal{H})$. Let $(M_i)_{i\in I}$ be a direct system in $\mathcal{H}$ and let $X\in\mathcal{H}\cap\mathcal{D}^b_{fg}(R)$ be any object. We consider the canonical morphism 
$$\eta_{X}:\varinjlim\text{Hom}_{\mathcal{D}(R)}(X,M_i)\longrightarrow\text{Hom}_{\mathcal{D}(R)}(X,\varinjlim_\mathcal{H}M_i), $$ which is a morphism in $\text{R}-\text{Mod}$. Localization at any prime ideal $\mathbf{p}$ preserves direct limits and, by \cite{PS3}[Proposition 3.11], we also have that $(\varinjlim_\mathcal{H}M_i)_\mathbf{p}\cong\varinjlim_{\mathcal{H}_\mathbf{p}}(M_i)_\mathbf{p}$. Here if $\mathcal{H}=\mathcal{H}_\phi$, then we put $\mathcal{H}_\mathbf{p}=\mathcal{H}_{\phi_\mathbf{p}}$, using the terminology of \cite{PS3}. Therefore, using Lemma \ref{lem.Hom under localization},  we can identify $(\eta_X)_{\mathbf{p}}:(\varinjlim\text{Hom}_{\mathcal{D}(R)}(X,M_i))_\mathbf{p}\longrightarrow (\text{Hom}_{\mathcal{D}(R)}(X,\varinjlim_\mathcal{H}M_i))_\mathbf{p}$ with the canonical morphism $$\eta_{X_\mathbf{p}}:\varinjlim\text{Hom}_{\mathcal{D}(R_\mathbf{p})}(X_\mathbf{p},(M_i)_\mathbf{p})\longrightarrow\text{Hom}_{\mathcal{D}(R_\mathbf{p})}(X_\mathbf{p},\varinjlim_{\mathcal{H}_\mathbf{p}}(M_i)_\mathbf{p}). $$ But the sp-filtration $\phi_\mathbf{p}$ of $\text{Spec}(R_\mathbf{p})$ also satisfies the weak Cousin condition and, since $R_\mathbf{p}$ has finite Krull dimension, we get that $\phi_\mathbf{p}$ is eventually trivial (see \cite{AJS}[Corollary 4.8]). The truth of the theorem when the associated filtration is eventually trivial implies that $\eta_{X_\mathbf{p}}$ is an isomorphism, for all $\mathbf{p}\in\text{Spec}(R)$, because $X_\mathbf{p}\in\text{fp}(\mathcal{H}_\mathbf{p})$. Therefore the kernel and cokernel of $\eta_X$ are $R$-modules with empty support. Then they are both zero, so that $\eta_X$ is an isomorphism, and hence $X$ is in $\text{fp}(\mathcal{H})$ as desired. 
\end{proof}

\begin{cor} \label{cor.realisation}
Let $R$ be a commutative noetherian ring. The heart of any t-structure in $\mathcal{D}^b_{fg}(R)$ is equivalent to the category of finitely presented objects of a locally coherent Grothendieck category. 
\end{cor}
\begin{proof}
Each t-structure  in $\mathcal{D}^b_{fg}(R)$ is the restriction of the t-structure $\tau_\phi$ in $\mathcal{D}(R)$ associated to an sp-filtration (see \cite[Corollary 3.12]{AJS}). The result is then an immediate consequence of last theorem, using \cite[Theorem 3.10]{AJS}.
\end{proof}

As a final comment, we give the geometric translation of last theorem and corollary:

\begin{cor} \label{cor.geometric translation}
Let $\mathbb{X}$ be an affine noetherian scheme and let $(\mathcal{U},\mathcal{U}^\perp [1])$ be a t-structure in $\mathcal{D}(\mathbb{X}):=\mathcal{D}(\text{Qcoh}(\mathbb{X}))$ which restricts to $\mathcal{D}^b_{coh}(\mathbb{X})\cong D^b(\text{coh}(\mathbb{X}))$. The heart $\mathcal{H}$ of the t-structure is a locally coherent Grothendick category on which $\mathcal{H}\cap\mathcal{D}^b_{coh}(\mathbb{X})$ is the class of finitely presented objects. In particular, the heart of each t-structure in $\mathcal{D}^b(\text{coh}(\mathbb{X}))$ is equivalent to the category of finitely presented objects of a locally coherent Grothendieck category. 
\end{cor}

\end{document}